\documentclass[12pt]{amsart}
\usepackage{amssymb, amsmath}
\usepackage{graphics}
\usepackage{epsfig}
\usepackage{amsxtra}
\usepackage{color}

\numberwithin{equation}{section}

\newtheorem{thm}{Theorem}[section]

\newtheorem{lemma}[thm]{Lemma}

\newtheorem{prop}[thm]{Proposition}
\newtheorem{cor}[thm]{Corollary}

\theoremstyle{definition}

\newcommand{\defeq}{\stackrel{\rm{def}}{=}}

\renewcommand{\Re}{\operatorname{\rm Re}\nolimits}
\renewcommand{\Im}{\operatorname{\rm Im}\nolimits}

\def \comp {\operatorname{comp}}
\def \supp {\operatorname{supp}}
\def \diam {\operatorname{diam}}
\def \dist {\operatorname{dist}}

\def \Vol {\operatorname{Vol}}

\def \vol {{\rm vol}}

\def \mfm {{\mathfrak M}}

\def \Real {{\mathbb R}}

\def \Complex {\mathbb{C}}

\def \meas {\operatorname{meas}}

\title []
{Schr\"odinger operators and the distribution of resonances in sectors
}
   \author { T.J. Christiansen}
\address{Department of Mathematics,
University of Missouri,
Columbia, Missouri 65211, USA} 
\email{christiansent@missouri.edu}
\begin{document}

\begin{abstract}
The purpose of this paper is to give some refined results about the distribution 
of resonances in potential scattering.  We use techniques 
and results from one 
and several complex variables, including properties of functions of completely
regular growth.  This enables us to find asymptotics for the distribution
of resonances in sectors for certain potentials and for
certain families of potentials.
\end{abstract}

\maketitle

\section{Introduction}
\label{s:int}

The purpose of this paper is to prove some results about the distribution 
of resonances in potential scattering.  In particular, we study the 
distribution of resonances in sectors and give
asymptotics of
the ``expected value'' of the number of resonances in certain settings.

More precisely, 
we consider the operator $-\Delta +V$, where 
$V\in L^{\infty}_{\comp}(\Real^d)$ and $\Delta $ is the (non-positive) 
Laplacian.  Then, except for a finite number of values of $\lambda$, $
R_V(\lambda)= (-\Delta +V-\lambda^2)^{-1},$ 
$\Im \lambda >0$, is a bounded  operator on $L^2(\Real^d)$ for
$\lambda $ in the upper half plane.  When $d$ is odd and 
$\chi\in L_{\comp}^{\infty}(\Real^d)$ satisfies 
$\chi V=V$,  $\chi R_V(\lambda) \chi$ has a meromorphic continuation to 
the lower half plane.  The poles of $\chi R_V(\lambda) \chi$ are called
{\em resonances}, and are independent of choice of $\chi$ 
satisfying these hypotheses.  
Resonances are analogous to eigenvalues 
not only in their
appearance as poles of the resolvent, but also because they appear in 
trace formulas much as eigenvalues do \cite{b-g-r,gu-zwan,sttwg}. 
Physically, they may be thought of as corresponding to decaying waves.

Let $n_V(r)$ denote the number of resonances of $-\Delta +V$, counted
with multiplicity, with norm at most $r$.  When $d=1$, asymptotics of
$n_V(r)$ are known:
$$\lim _{r\rightarrow \infty}\frac{n_{V }(r)}{r}
=\frac{2}{\pi }\diam (\supp(V))
$$
\cite{zw1}; see also
\cite{froese1, regge,simon}.  
Moreover, ``most'' of the resonances occur in sectors about the 
real axis, in the sense that for any $\epsilon>0$, 
$$\lim_{r\rightarrow \infty}\frac{ \# \{ \lambda_j\;
 \text{pole of $R_V(\lambda)$}:
|\arg \lambda_j -\pi|<\epsilon\; \text{or } |\arg \lambda_j-2\pi|<\epsilon
\}}{r}=\frac{2}{\pi}\diam (\supp (V)),$$
\cite{froese1}.
These results are valid for complex-valued $V$.
The case $d=1$ is exceptional, though:
in higher dimensions much less is known.

Now we turn to $d\geq 3 $ odd, where the question is more subtle.
If $V\in L^{\infty}(\Real^d)$ has support in $\overline{B}(0,a)=
\{ x\in \Real^d:\; |x|\leq a\}$, then 
\begin{equation}
\label{eq:stupperbd}
d\int_0^r \frac{n_V(t)-n_V(0)}{t}dt\leq c_d a^d r^d +o(r^d).
\end{equation}
where $c_d$ is defined in (\ref{eq:cd}) and depends only on the dimension.
Zworski \cite{zwodd} showed that such a bound holds, and 
Stefanov \cite{stefanov} identified the optimal constant.
There are examples for which equality holds in (\ref{eq:stupperbd}),
\cite{zwradpot, stefanov}.
Lower bounds have proved more elusive.  The current best known general
quantitative lower bound is for non-trivial real-valued 
$V\in C_c^{\infty}(\Real)$ 
 \begin{equation}\label{eq:>0}
\lim \sup _{r\rightarrow \infty}\frac{n_{V}(r)}{r}>0
\end{equation}
\cite{sBlb}.
On the other hand, there are nontrivial complex-valued
 potentials $V$ for which $\chi R_V(\lambda)\chi$ has no poles, 
\cite{counterex}.  

We wish to single out the set for which asymptotics actually hold
in (\ref{eq:stupperbd}).   This is the set defined, for 
 $a>0$, as
\begin{equation}\label{eq:mv}
\mfm_a = \{ V\in L^{\infty}(\Real^d): \; \supp V \subset 
\overline{B} (0,a)
\; \text{and} \;
n_V(r) = c_d a^dr^d+o(r^d) \; \text{as $r\rightarrow \infty$} \}.
\end{equation}
We remark that it is equivalent to require, as $r\rightarrow \infty$,
$n_V(r) = c_d a^dr^d+o(r^d)$ or $d\int_0^rt^{-1}(n_V(t)-n_V(0)) \; dt = c_d a^d r^d
+ o(r^d)$.  
The set $\mfm_a$ contains infinitely many radial potentials.  By results of 
\cite{zwradpot, stefanov}, this set contains any potential of the 
form $V(x)=v(|x|)$, where $v\in C^2([0,a])$ is 
real-valued, $v(a)\not = 0$, and 
$V(x)=0$ for $|x|>a$.  Additionally, it contains infinitely
many complex-valued potentials which are isoresonant with these 
real-valued radial potentials \cite{isophasal}. 

We now can state some results.
For the first, we set, for $\varphi<\theta$,
$n_V(r,\varphi,\theta)$ to be the set of poles of $R_V(\lambda)$,
counted with multiplicity, with 
norm at most $r$ and with argument between $\varphi$ and $\theta$ inclusive.
 \begin{prop}\label{p:angles}
Let $V\in \mfm_a$.  Then, if $0<\varphi <\theta< \pi$,
$$n_V(r,\pi +\varphi,\pi+\theta)= \frac{1}{2\pi d}
\tilde{s}_d(\varphi, \theta)r^d a^d+ o(r^d)\; 
\text{as $r\rightarrow \infty$}$$
where 
$$\tilde{s}_d(\varphi,\theta)=h_d'(\theta)-h_d'(\varphi)+
d^2\int_{\varphi}^{\theta} h_d(s) ds,$$
and $h_d(\theta)$ is as defined in (\ref{eq:hd}).
\end{prop}
If $V$ is real-valued, then $\lambda_0$ is a resonance 
of $-\Delta +V$ if and only if 
$-\overline{\lambda_0}$ is a resonance.  In this case for
$V\in \mfm_a$ and  $0<\theta<\pi$ 
\begin{equation}
\label{eq:nearpi}
n_V(r,\pi,\pi+ \theta)=\frac{1}{2\pi d}\left[
h_d'(\theta)+ d^2\int_{0}^{\theta} h_d(s) ds\right] a^dr^d
+o(r^d).
\end{equation}
Here, as elsewhere in this paper, we are concerned with the behavior as
$r\rightarrow \infty$.  Thus, one should understand that statements of
the type $f(r)=g(r)+o(r^p)$ are statements which hold for 
$r$ sufficiently large.

Corollary \ref{cor:nearreal} shows that (\ref{eq:nearpi})
 holds for any $V\in \mfm_a$.
These results show that any potential 
in $\mathfrak{M}_a$ must have resonances distributed regularly in sectors,
as well as being distributed regularly in balls centered at the origin.
A result like this proposition and Corollary \ref{cor:nearreal} is,
 for the special potentials of the form $V(x)=v(|x|)$ mentioned
earlier, implicit in the papers of Zworski \cite{zwradpot}
and Stefanov \cite{stefanov}.  Here we derive it as a corollary of some 
complex-analytic results, and note that it holds for {\em any} potential 
$V\in \mfm_a$.  We note that this proposition could, in fact, follow as 
a corollary to Theorem \ref{thm:second}.  However, we prefer to give a 
separate proof that uses standard results for functions of completely
regular growth.  

In the following theorem we use the notation
$N_V(r)=\int_0^r \frac{1}{t}(n_V(t)-n_V(0))dt$.  This theorem shows that
there are many potentials for which something close to the optimal upper
bound on the resonances is achieved.
\begin{thm}\label{thm:first}
Let $\Omega \subset \Complex^p$ be an open connected set.  
Suppose $V(z)=V(z,x)$ 
is holomorphic in $z\in\Omega$ and,
 for each $z\in \Omega$, $V(z,x)\in L^{\infty}(\Real^d)$,
and $
V(z,x)=0$ if $|x|>a$.  Suppose in addition that for some $z_0\in 
\Omega $, $V(z_0)\in \mfm_a$.  Then there is a pluripolar set $E\subset \Omega$
so that $$\lim \sup _{r \rightarrow \infty} \frac{N_{V(z)}(r)}{r^d}
= \frac{c_d a^d}{d} \; \text{for all}\; z\in \Omega \setminus E.$$
Moreover, for any $\theta>0$, $\theta<\pi$, there is 
a pluripolar set $E_\theta$ so that 
$$\lim \sup _{r \rightarrow 
\infty} \frac{N_{V(z)}(r,\pi,\pi+\theta)}{r^d} \geq \lim_{\epsilon \downarrow 0}
\frac{a^d}{4\pi d^2} h_d'(\epsilon)
$$
for all $z\in \Omega \setminus E_\theta$.
\end{thm}
For example, one may take, for 
$z\in \Complex$, $V(z)=zV_1+(1-z)V_0$, where $V_0\in \mfm_a$ and
$V_1\in L^{\infty}(\Real^d)$ has support in $\overline{B}(0,a)$.  
Since $h_d'(0+)=\lim_{\epsilon \downarrow 0}h_d'(\epsilon)>0$, see 
Lemma \ref{l:hdbasics}, the second statement of the theorem is meaningful.
This result is of particular interest since resonances near the real 
axis are considered the more physically relevant ones.

We recall 
the definition of a pluripolar set in Section \ref{s:complex}.
Here we mention
that a 
pluripolar set is small.  A pluripolar 
set $E\subset \Complex^p$   has $\Real^{2p}$ Lebesgue measure $0$, and 
if $E\subset \Complex $ is pluripolar, $E\cap \Real$ has one-dimensional
Lebesgue measure $0$ 
( see, for example, \cite{l-g, ransford}).   Thus the statements of
Theorem \ref{thm:first} 
hold for ``most'' values of $z\subset \Omega$.

If we take a weighted average over a family of potentials,
a kind of expected value, we are able to 
find asymptotics analogous to those which hold for a potential 
in $\mfm_a$.
In the statement of the next theorem and later in the paper, we 
use the notation $d\mathcal{L}(z)= d\Re z_1 d\Im z_1\cdot\cdot \cdot 
d\Re z_p d\Im z_p$. 
\begin{thm} \label{thm:second}
 Suppose the hypotheses of Theorem \ref{thm:first} 
are satisfied.  Then 
for any $\psi \in C_c(\Omega)$, 
$$\int_{\Omega} \psi(z) n_{V(z)}(r) d\mathcal{L}(z) =c_d a^d r^d
 \int_{\Omega} \psi(z) d\mathcal{L}(z)
 +o(r^d)$$
as $r\rightarrow \infty$.
Additionally, we have,
for  $0<\varphi <\theta< \pi$,
$$\int_\Omega
\psi(z)n_{V(z)}(r,\varphi+\pi,\theta+\pi)d\mathcal{L}(z)=
 \frac{1}{2\pi d}\tilde{s}_d(\varphi, \theta )r^d a^d 
\int_\Omega \psi(z)d\mathcal{L}(z)+ o(r^d)$$
where $\tilde{s}_d$ is as defined in Proposition \ref{p:angles}.
Moreover,
 for
$0<\theta<\pi$, 
\begin{multline*}
\int_{\Omega}\psi(z) n_{V(z)}(r,\pi,\theta+\pi)d\mathcal{L}(z) \\
= 
\frac{1}{2\pi d}\left[h_d'(\theta)+ d^2\int_{0}^{\theta} h_d(s) ds\right]a^dr^d 
\int_{\Omega}\psi(z)d\mathcal{L}(z)
+o(r^d).
\end{multline*}
\end{thm}

\begin{cor}\label{cor:nearreal}
Let $V\in \mfm_a$.  For any $0<\theta<\pi$, 
\begin{equation}\label{eq:angle1}
n_V(r,\pi,\theta+\pi) = \frac{1}{2\pi d}\left[
h_d'(\theta)+ d^2\int_{0}^{\theta} h_d(s) ds\right] a^dr^d
+o(r^d)
\end{equation}
and, for any $0<\varphi<\pi$,
\begin{equation}\label{eq:angle2}
n_V(r,\varphi+\pi,2\pi)=\frac{1}{2\pi d}\left[-h_d'(\varphi)+ d^2\int_{\varphi}^{\pi} h_d(s) ds\right]
a^dr^d
+o(r^d)
\end{equation}
as $r\rightarrow \infty$.
\end{cor}
This corollary follows from Theorem \ref{thm:second} by taking $V(z)$ equal
to the constant (in $z$) potential $V$.  We could instead give a more 
direct proof by, essentially, simplifying the proof of Proposition
\ref{p:avg1} and then applying Lemma \ref{l:noavg}.

It is worth noting that the coefficients of $r^d$ in (\ref{eq:angle1}) 
and (\ref{eq:angle2}) are positive, so that in any sector in the lower half
plane which touches the real axis, the number of resonances grows like $r^d$.

The proofs of the results here are possible because of the optimal 
upper bounds on 
$\lim\sup _{r\rightarrow \infty}r^{-d} \ln |\det S_V(re^{i\theta})|$,
$0<\theta<\pi$, proved in \cite{stefanov}, see Theorem \ref{t:stefanov} here.
These, combined with some one-dimensional complex analysis, are used to 
prove Proposition \ref{p:angles}, and could be used to give a direct proof 
of Corollary \ref{cor:nearreal}.  The proofs of the other theorems 
use, in addition
to one-dimensional complex analysis, some facts about plurisubharmonic 
functions.  Many of the complex-analytic results which we shall use are 
recalled in Section \ref{s:complex}.

Again, we emphasize that we are concerned here with large $r$ behavior
of resonance counting functions, and consequently of other functions
as well.  Thus, statements of the type $f(r)=g(r)+o(r^p)$
are to be understood as holding for large values of $r$.

{\bf Acknowledgments.} It is a pleasure
to thank Plamen Stefanov and Maciej Zworski
for helpful conversations during the writing of this paper, and the
referee for helpful comments which improved the exposition.
  The author 
gratefully acknowledges the partial support of the NSF under grant 
DMS 1001156.

\section{Some complex analysis}
\label{s:complex}

In this section we recall some definitions 
and results from complex analysis in one
and several variables.  We will mostly follow the 
notation and conventions of \cite{levin} and \cite{l-g}.  We also 
prove a result, Proposition \ref{p:showcrg}, for which we are unaware of
a proof in the literature.

The {\em upper relative measure} of a set $E \subset \Real_+$ 
is 
$$\lim\sup_{r\rightarrow \infty} \frac{\meas(E\cap(0,r))}{r}.$$
A set $E\subset \Real_+$ is said to have {\em zero relative measure}
if it has upper relative measure $0$. 

If $f$ is a function holomorphic in the sector $\varphi<\arg z <\theta$,
we shall say $f$ is of order $\rho$ if $\lim \sup _{r \rightarrow \infty}
\frac{\ln \ln (\max_{\varphi<\phi < \theta}|f(r e^{i\phi})|)}{\ln r}
=\rho $.  We shall further restrict ourselves to functions of order $\rho$
and 
{\em finite type}, so that 
$$\lim \sup _{r \rightarrow \infty}
\frac{\ln (\max_{\varphi<\phi < \theta}|f(r e^{i\phi})|)}{r^{\rho}}<\infty.$$
We shall use similar definitions for a function holomorphic in a neighborhood
of a closed sector.  In this section only, we shall, following 
Levin \cite{levin}, use the notation $h_f$ for the
{\em  indicator function} (or {\em indicator})
of a function $f$ of order $\rho$: 
$$h_{f}(\theta)\defeq \lim \sup 
_{r\rightarrow \infty} \left( r^{-\rho}\ln |f(r e^{i\theta})|\right).$$
Suppose $f$ is a function analytic in the angle $(\theta_1,\theta_2)$ 
and of order $\rho$  and finite type there.  
The function $f$ is of {\em completely 
regular growth} on some set of rays $R_\mathfrak{M}$ ($\mathfrak{M}$
is the set of values of $\theta$) if the function
$$h_{f,r}(\theta) \defeq \frac{\ln |f(r e^{i\theta})|}{r^{\rho}}$$
converges uniformly to $h_f(\theta)$ for $\theta\in \mathfrak{M}$ 
when $r$ goes to infinity taking on all positive
values except possibly for a set $E_{\mathfrak{M}}$ of zero relative measure.
The function $f$ is of {\em completely regular growth in the angle
$(\theta_1,\theta_2)$ } if it is of completely
regular growth on every closed interior angle.


Functions of completely regular growth have zeros that are rather 
regularly distributed. 
For a function $f$ holomorphic in $\{z: \theta_1<\arg z < \theta_2\}$
we define, for $\theta_1<\varphi<\theta<\theta_2$,
$m_f(r,\varphi,\theta)$ to be the number of zeros of $f(z)$ in the 
sector $\varphi\leq \arg z \leq \theta$, $|z|\leq r$.
\footnote{More standard notation would be $n(r,\varphi,\theta)$, but
we have already defined $n_V(r,\varphi,\theta)$ to be something else.} 
 We recall the following theorem from \cite{levin}.
\begin{thm}\cite[Chapter III, Theorem 3]{levin} \label{thm:crg}
If a holomorphic function $f(z)$ of order $d$ and finite type
has completely regular 
growth within an angle $(\theta_1, \theta_2)$, then for all values of 
$\varphi$ and $\theta$, ($\theta_1< \varphi <\theta <\theta_2$ )
except possibly for a denumerable set, the following limit will exist:
$$\lim_{r \rightarrow \infty} \frac{m_f(r,\varphi, \theta)}{r^d} 
= \frac{1 }{2 \pi d}\tilde{s}_f(\varphi, \theta)$$
where 
$$ \tilde{s}_f(\varphi,\theta)= \left[ h_f'(\theta)-h_f'(\varphi)
+ d^2 \int_\varphi^\theta h_f(s)ds\right].$$
The exceptional denumerable set can only consist of points for which 
$h_f'(\theta + 0)\not = h_f'(\theta - 0)$.
\end{thm}

In the following proposition we use the notation $m_f(r)$ to 
denote the number of zeros of a function $f$, 
counted with multiplicity, with norm at most $r$.  It is likely
that some of the hypotheses included here could be relaxed.  However,
when we apply this proposition, $f$ will be the determinant of the 
scattering matrix, perhaps multiplied by a rational function, and 
many of these hypotheses are natural in such applications.

Let $f(z)$ be a function meromorphic on $\Complex$.  Then 
$f(z)=g_1(z)/g_2(z)$, with $g_1,$ $g_2$ entire.  The functions 
$g_1$ and $g_2$ are not uniquely determined.  
However, the order of $f$ can be defined to be
$$\min\{ \max(\text{order of $g_1$, order of $g_2$): 
$f(z)=g_1(z)/g_2(z)$ with $g_1$, $g_2$ entire}\}.$$
It is possible to define the order of a meromorphic function by 
using the Nevanlinna characteristic function, yielding the same result.
\begin{prop}\label{p:showcrg}
Let $f$ be a function meromorphic in the complex plane, with neither zeros
nor poles on the real line.
Suppose all the zeros of $f$ lie in 
the open upper half plane, and all the poles in the open lower half plane.  
Furthermore, assume $f$ is of order $d>1$, 
$h_f$ is finite for $0\leq \theta\leq \pi$, and $h_f(\theta_0)\not =0$
for some $\theta_0$,  $0<\theta_0<\pi$.  Suppose in addition
\begin{equation}\label{eq:argbd}
\int _0^r \frac{f'(t)}{f(t)}dt =o( r^d) \; \text{as $r\rightarrow \pm \infty$},
\end{equation}
and the number of poles of $f$ with norm at most $r$ is of order at most
$d$.
If 
 $$\lim \inf _{r\rightarrow 
\infty} \frac {m_f(r)} {r^d} = \frac{d}{2\pi} \int_0^{\pi} h_f(\theta)
d\theta,$$
then $f$ is of completely regular growth in the angle $(0,\pi)$.
\end{prop}
Before proving the proposition, we note that Govorov \cite{g1,g2} 
has studied the issue of completely regular growth of functions holomorphic
in an angle. This is discussed in \cite[Appendix VIII, section 2]{levin}.
This is somewhat different than what we consider, since we use the 
assumption that $f$ is meromorphic and of order $d$ on the plane.  
Thus Govorov uses different restrictions on the distribution of 
the zeros of $f$.
\begin{proof}
The proof of this proposition follows in outline the proof of the analogous
theorem for entire functions in the plane, \cite[Chapter IV, Theorem 3]{levin}.
Rather than using Jensen's theorem, though, it uses the equality
\begin{equation}\label{eq:frid}
\int_0^r \frac{m_f(t)}{t}dt = \frac{1}{2\pi} \Im \int_{0}^r \frac{1}{t}
\int_{-t}^t
\frac{f'(s)}{f(s)} ds dt + 
\frac{1}{2\pi} \int_0^{\pi} \ln |f(r e^{i\theta})|d\theta 
\end{equation}
if $|f(0)|=1$,
which follows using the proof of \cite[Lemma 6.1]{froese}.

By \cite[Property (4), Chapter I, section 12]{levin}, 
\begin{equation}\label{eq:lowerbdlevin}
\lim\inf_{r\rightarrow \infty}\frac{ m_f(r)}{r^d}
\leq \lim\inf _{r\rightarrow \infty}d r^{-d} \int_0^r \frac{m_f (t)}{t} dt.
\end{equation}
We note \cite[Chapter I, Theorem 28]{levin}
that for any $\epsilon >0$ there is an $R>0$ so that
\begin{equation}
\label{eq:asympteq}
r^{-d}\ln|f(r e^{i\theta})| \leq h_f(\theta)+\epsilon,\; 
\text{for $r>R$, $0\leq \theta \leq \pi$}.
\end{equation} 
Using this,  (\ref{eq:frid}), and our assumptions on the behavior of $f$ on
the real axis,
we see that 
$$\lim\sup_{r \rightarrow \infty} r^{-d}
\int_0^r \frac{m_f(t)}{t}dt \leq \frac{1}{2\pi}
\int_0^{\pi} h_f(\theta)d\theta.$$

Combining this with (\ref{eq:lowerbdlevin}) and using our assumptions on 
$m_f(r)$, we get 
$$\lim _{r\rightarrow \infty} r^{-d}\int_0^r\frac{m_f(t)}{t}dt = \frac{1}{2\pi}
\int_0^\pi h_f(\theta)d\theta.$$
Thus using (\ref{eq:frid}) and (\ref{eq:argbd}) again, we have 
$$\lim _{r \rightarrow \infty} \int_0^\pi [h_f(\theta) 
- r^{-d}\ln|f(r e^{i\theta})|]d\theta =0,$$
and, using (\ref{eq:asympteq}), 
$$\lim _{r \rightarrow \infty} \int_0^\pi \left |h_f(\theta) 
- r^{-d}\ln|f(r e^{i\theta})|\right|d\theta =0.$$


Since we have assumed $f$ is of order $d$, we may write
$f$ as the quotient of two entire functions, each of order at most $d$.  Then
we may apply \cite[Chapter 2, Theorem 7]{levin} to find that for 
every $\eta>0$ there is a set $E_\eta$ of 
positive numbers of upper relative measure less than $\eta$
so that if $r\not \in E_\eta$, the family of functions of $\theta$,
$$h_{f,r}(\theta)\defeq r^{-d} \ln |f(r e^{i\theta})|,$$
is equicontinuous in the angle $0<\epsilon_0 \leq \theta \leq \pi-\epsilon_0.$

Given $\eta>0$ and $\epsilon>0$ we can,
by the above result, find a $\delta >0$ with $(\theta_1-\delta,\theta_2+\delta)
\subset (0,\pi)$ and
a set $E_\eta$ of upper relative measure at 
most $\eta$ so that if
$\theta\in (\theta_1,\theta_2)$,
 $r\not \in E_\eta$, and $|\varphi-\theta|<\delta$, then 
$|h_{f,r}(\theta)-h_{f,r}(\varphi)|<\epsilon/4$
and $|h_f(\theta)-h_f(\varphi)|
<\epsilon/4$.  Then for $0<|k|<\delta$, $r\not \in E_\eta$,
\begin{align*}
|h_{f,r}(\theta)-h_{f}(\theta)|& < \epsilon/2 +\frac{1}{k} 
\int_\theta^{\theta+k} |h_{f,r}(\varphi)-h_f(\varphi)|d\varphi\\
& \leq \epsilon/2+ \frac{1}{k}\int_0^\pi |h_{f,r}(\varphi)-h_f(\varphi)|d\varphi.
\end{align*}
Since the 
integral goes to $0$ as $r\rightarrow \infty$, we have shown
that for $r > r_{\epsilon}$, $r \not \in E_{\eta}$, 
$|h_{f,r}(\theta)-h_{f}(\theta)|<\epsilon$.  Since $\eta>0 $ and $\epsilon>0$
are arbitrary, we have, by \cite[Chapter III, Lemma 1]{levin}, $f$ is of completely
regular growth in $(\theta_1,\theta_2)$.
\end{proof}

We shall also need some basics about plurisubharmonic functions and 
pluripolar sets.  We use notation as in \cite{l-g} and refer the 
reader to this reference for more details.

Let $\Omega \subset \Complex^p$ be an open connected set.  
A function $\Psi:\Omega \rightarrow [-\infty, \infty)$ is said to be 
{\em plurisubharmonic } if  $\Psi \not \equiv -\infty$,
$\Psi$ is upper semi-continuous, and 
$$\Psi(z)\leq \frac{1}{2\pi} \int_0^{2\pi} \Psi(z+w r e^{i\theta})d\theta$$
for all $w,\; r$ such that $z+uw\subset \Omega$ for all $u\in \Complex, $ 
$|u|\leq r$.   The classic example of a plurisubharmonic function is
$\ln |f(z)|$, where $f(z)$ is holomorphic.  
A subset $E\subset \Omega\subset \Complex^p$ is said to be {\em pluripolar}
if there is a function $\Psi$ plurisubharmonic on $\Omega$ so that 
$E\subset \{ z: \psi(z)= -\infty\}.$

For the convenience of the reader, we recall \cite[Proposition 1.39]{l-g},
which is the main additional fact from several complex variables which we 
shall need.
\begin{prop}\label{p:l-g} (\cite[Prop. 1.39]{l-g}) Let $\{\Psi_q\}$
be a sequence of plurisubharmonic functions uniformly bounded above on
compact subsets in an open connected set $\Omega \subset \Complex^p$,
with $\lim \sup _{q\rightarrow \infty} \Psi_q \leq 0$
and suppose that there exist $\xi \in \Omega$ such that 
$\lim \sup _{q\rightarrow \infty}\Psi_q(\xi)=0$.
Then $A=\{ z\in \Omega: \lim \sup_{q\rightarrow \infty}\Psi_q(z)<0\}$ 
is pluripolar in $\Omega$.
\end{prop}

\section{The functions $s_V(\lambda)=\det S_V(\lambda)$ and 
$h_d(\theta)$}

For $V\in L^{\infty}_{\comp}(\Real^d)$ and $\chi\in L^{\infty}_{\comp}(\Real^d)
$ with $\chi V=V$, we have
$\chi R_V(\lambda)\chi= \chi R_0(\lambda)\chi(I+VR_0(\lambda)\chi)^{-1}$.
Since for any $\chi$ with compact support in $\Real^d$, $\| \chi
R_0(\lambda)\chi\|\leq  c_{\chi}/|\lambda|$ when $\Im \lambda \geq 0$, 
we see that $R_V(\lambda)$ can have only finitely many poles in the 
closed upper half plane.  

For $V\in L^{\infty}_{\comp}(\Real^d)$, let $S_V(\lambda)$ be the associated
scattering matrix and $s_V(\lambda)=\det S_V(\lambda)$.  With at most
finitely many exceptions, the poles of $s_V(\lambda)$ coincide with 
the poles of $R_V(\lambda)$, and the multiplicities agree.  Moreover,
$s_V(\lambda)s_V(-\lambda)=1$.  
We recall \cite[Lemma 3.1]{polar}:
\begin{lemma}\label{l:derivest}
Let $V\in L^{\infty}_{\comp}(\Real^d;\Complex)$.
 For $\lambda \in \Real$, there is a $C_{V}$
so that
$$\left| \frac{d}{d\lambda}\ln s_{V}(\lambda)\right|
\leq C_{V}|\lambda|^{d-2}$$
whenever $|\lambda|$ is sufficiently large.
\end{lemma}
In fact,  if 
$\supp V\subset \overline{B}(0,a)$ there is a constant 
$\alpha_d=\alpha_{d,a}$ so that  it suffices to take 
$|\lambda |\geq 2 \alpha_d \|V\|_{\infty}$ 
for such a bound to hold.
We note that for $\lambda \in \Real$, $|\lambda|\geq 2 \alpha_d \|V\|_\infty$
under these same assumptions on $V$,
\begin{equation}\label{eq:simplebd}
\| S_V(\lambda)-I\| \leq C |\lambda |^{-1}.
\end{equation}  This is relatively easy to see from an explicit 
representation of the scattering matrix;
see, for example, the proof of \cite[Lemma 3.1]{polar}.  The constants in the
statement of \cite[Lemma 3.1]{polar} and in 
(\ref{eq:simplebd}) can be chosen to depend only on 
the dimension, $\|V\|_{\infty}$ and the support of $V$.
We note that it follows from Lemma \ref{l:derivest}, (\ref{eq:simplebd}),
and (\ref{eq:frid}) that as $r \rightarrow \infty$
\begin{equation}\label{eq:almostsame}
\int_0^r \frac{n_V(t)}{t} dt
= \int_0^\pi \ln |\det S_V(r e^{i\theta})| d\theta +O(r^{d-1}). 
\end{equation}

Let
\begin{equation}\label{eq:rho}
\rho(z)\defeq \ln \frac{1+ \sqrt{1-z^2}}{z}-\sqrt{1-z^2}, \; 0<\arg z <\pi.
\end{equation}
This
is a function which arises in studying the asymptotics of Bessel functions; see
\cite{olverrs}.  To define the square root which appears here, take 
the branch cut on the negative real axis and define $\rho$ to 
be a continuous function in $\{ 0<\arg z<\pi\} \cup (0,1)$
and use the 
principal branches of the logarithm and the square root when $z\in(0,1)$.

We use some notation of \cite{stefanov}.  Set, for $0 < \theta<\pi$, 
\begin{equation}\label{eq:hd}
h_d(\theta)\defeq \frac{4}{(d-2)!}\int_0^{\infty} \frac{[-\Re \rho]_+(t e^{i\theta})}
{t^{d+1}} dt
\end{equation}
and set $h_d(0)=0$, $h_d(\pi)=0$.
Now set
\begin{equation}\label{eq:cd}
c_d 
\defeq \frac{d}{2\pi}\int_0^\pi h_d(\theta) d\theta = 
\frac{2d}{\pi (d-2)!}\int _{\Im z>0 }\frac{[-\Re \rho]_+(z)}{|z|^{d+2}} dx dy.
\end{equation}
This is the constant $c_d$ which appears in (\ref{eq:stupperbd}).

We recall the following result of \cite[Theorem 5]{stefanov}, which
we paraphrase to suit our setting; \cite[Theorem 5]{stefanov}
actually covers a much larger class of operators.
\begin{thm}(from \cite[Theorem 5]{stefanov}) \label{t:stefanov}
 Let $V\in L^{\infty}(\Real^d)$
be supported in $\overline{B}(0,a)$.\\
(a) For any $\theta\in [0,\pi], $
\begin{equation}
\ln|s_V(r e^{i \theta})|\leq h_d(\theta ) a^d r^d+o(r^d)\; \text{as $r\rightarrow 
\infty$},
\end{equation}
and the remainder term depends on $V$, and is uniform for $0<\delta \leq
\theta\leq \pi -\delta$ for any $\delta \in (0,\pi)$.
\\
(b) For any $\delta>0$, 
$$\ln|s_V(r e^{i\theta})|\leq 
(h_d(\theta)a^d+\delta)r^d+o(r^d)\; \text{as $r\rightarrow 
\infty$}$$
uniformly in $\theta\in [0,\pi]$.
\end{thm}
We remark that both of these statements are about ``large $r$'' behavior,
so that the possibility that $s_V$ has a finite number of poles in the 
upper half plane does not affect the validity of the statements.

It is important to note several things about the bounds in this
theorem.  One is that
although Stefanov's theorem is stated only
for self-adjoint operators (hence $V$ real)
it is equally valid when we allow complex-valued
potentials.  In fact, the proof of (a) in  
\cite[Theorem 5]{stefanov}
uses self-adjointness only to obtain a bound on the 
resolvent for $\lambda$ in the upper half plane.  A similar bound is
true for the operator $-\Delta + V$ when $V$ is complex-valued.
The proof of (b) uses the fact that for real $V$,
if $\lambda \in \Real$, $\ln |s_V(\lambda)|=1$.  For
complex-valued $V$, the proof in \cite{stefanov} of (b) can 
be adapted by using (\ref{eq:simplebd}) and Lemma \ref{l:derivest}
to show that for $\lambda \in \Real$, $|\lambda|\geq 2 \alpha_d \|V\|_{\infty}$,
$|\ln s_V(\lambda)| \leq C (1+|\lambda|)^{d-1}$.  Here $C$ can be chosen 
to depend only on
 $d$, $\|V\|_{\infty}$ and the diameter of the support of $V$.

Likewise,
the particulars of the operator enter only through the diameter of the 
support of the perturbation (for us, the diameter of
the support of $V$, which is $2a$) and the aforementioned bound on the 
resolvent in the good half plane $\Im \lambda>0$.  Thus, it is easy to see that 
the estimates of Theorem \ref{t:stefanov} are uniform in $V$ as
long as $\supp V \subset \overline{B}(0,a)$, $\|V\|_{\infty}\leq M$, 
and $r\geq 2 \alpha_d M$.

We note that the upper bound (\ref{eq:stupperbd}) 
on the integrated resonance-counting function holds with the
constant $c_d$ defined in (\ref{eq:cd}) even if $V$ is complex-valued.
This follows from the proof in \cite{stefanov}.  In fact, the 
proof uses the bounds recalled in Theorem \ref{t:stefanov}
and the identity (\ref{eq:frid}).  Together with
the bounds in Lemma \ref {l:derivest} and (\ref{eq:simplebd}), these 
prove (\ref{eq:stupperbd}), even when $V$ is complex-valued.

We shall want to understand the function $h_d(\theta)$ better.  Note that
for $0<\theta\leq \pi/2$, 
$$h_d(\pi/2+\theta)= h_d(\pi/2 -\theta).$$
This can be seen directly using the definition of $h_d$ and $\rho$.
\begin{lemma}\label{l:hdbasics}
The function $h_d(\theta)$, defined in (\ref{eq:hd}), is $C^1$ on
$(0,\pi)$.  Moreover, 
$$h_d'(0+) \defeq \lim_{\epsilon \downarrow 0} h_d'(\epsilon)
= \sqrt{\pi} \frac{ \Gamma\left( \frac{d-1}{2}\right)}
{(d-2)! \Gamma\left( 1+ \frac{d}{2}\right)}.
$$
\end{lemma}
\begin{proof}
We note \cite[Section 4]{olverrs} that 
$\Re \rho(z)<0 $ if $0<\arg z<\pi$ and $|z|>|z_0(\arg z)|$, where 
$z_0(\theta)$ is the
unique point in $\Complex$ with
 argument $\theta$ and which lies on the curve given by
$$ \pm (s \coth s  - s^2)^{1/2} + i(s^2 - s \tanh s)^{1/2}, 0 \leq s \leq  s_0. $$
Here $s_0$ is the positive solution of 
$\coth s =s. $
Furthermore, $\Re \rho(z)>0 $ if $z$ is in the upper half plane 
but $|z|<|z_0(\arg z)|.$
Hence, recalling the definition of $h_d$, we have
$$h_d(\theta)= \frac{4}{(d-2)!}\int_{|z_0(\theta)|}^{\infty} 
\frac{[-\Re \rho](t e^{i\theta})}
{t^{d+1}} dt.$$

Using the  definition of $\rho$ (\ref{eq:rho})
and the following comments, we see that $\rho$ 
is in fact a smooth function of $z$ with $0<\arg z<\pi$, $|z|>0$.
Since $|\rho(z)|/|z| \rightarrow 1$ when $|z| \rightarrow \infty$ in this
region, the integral defining $h_d$ is absolutely convergent.
Likewise, since 
$$\frac{\partial}{\partial \theta}\rho(t e^{i\theta})
= -i \sqrt{1-(t e^{i\theta})^2}$$
we have $$\left|\frac{-\Re \left[ \frac{\partial}{\partial \theta} \rho (t e^{i\theta}) \right]}
{t^{d+1}}\right| \leq C t^{-d}$$ and
the integral 
$$\int_{|z_0(\theta)|}^{\infty} 
\frac{-\Re \left[ \frac{\partial}{\partial \theta} \rho (t e^{i\theta}) \right]}
{t^{d+1}} dt$$
converges absolutely. 
A computation shows that $|z_0|$ is a $C^1$ function 
of $\theta$ for $\theta$ in $(0,\pi)$, and $\lim _{\epsilon \downarrow 0}
\frac{\partial}{\partial \theta}{|z_0|}$ is finite.
Thus, using that $\Re \rho(z_0(\theta))=0$ and the regularity of 
the derivative of $|z_0|(\theta)$, we get
$$\frac{d}{d\theta}h_d(\theta)=
\frac{4}{(d-2)!}\int_{|z_0(\theta)|}^{\infty} 
\frac{ \Re i \sqrt{1-(t e^{i\theta})^2  }} {t^{d+1}} dt
$$
which is continuous in $\theta$.
Thus $h_d$ is $C^1$ on $(0,\pi)$, 
$h_d'(0+)=\frac{4}{(d-2)!}\int_1^\infty \frac{\sqrt{t^2-1}}{t^{d+1}}dt,$
and a computation now finishes the proof of the lemma.
\end{proof}

If $d=3$, we can compute that
$$h_3(\theta)= \frac{4}{9} \left( \sin (3 \theta) + \Re 
\frac{ (1-z_0^2 (\theta))^{3/2}}
{|z_0(\theta)|^3 } \right)$$
where $z_0(\theta)$ is as in the proof of the lemma.  
We comment that the $\sin (3 \theta)$ term is missing from the first
remark following the statement of  \cite[Theorem 5]{stefanov}. 

 \section{Proof of Proposition \ref{p:angles}}

We can now give the proof of Proposition \ref{p:angles}, which 
follows by combining Theorem \ref{thm:crg}, Proposition \ref{p:showcrg},
and \cite[Theorem 5]{stefanov}.

Recall that $S_V(\lambda)$ is the scattering matrix associated with the operator
$-\Delta +V$, and $s_V(\lambda)=\det S_V(\lambda)$.  Then $s_V$ has a pole
at $\lambda$ if and only if $s_V$ has a zero at $-\lambda$, and the 
multiplicities coincide.  Moreover, with at most a finite number of 
exceptions, the poles of $s_V(\lambda)$ coincide, with multiplicity,
with the zeros of $R_V(\lambda)$.

If $s_V(\lambda)$ has poles in the closed 
upper half plane, it has only finitely
many, say $\lambda_1,...,\lambda_m$, where the poles
are repeated according to multiplicity.  Set  
$$f(\lambda)=\prod_{j=1}^m\frac{(\lambda -\lambda_j)}{\lambda+\lambda_j}
s_V(\lambda).$$  We check that 
$f$ satisfies the hypotheses of Proposition \ref{p:showcrg}.
Note that $f$ and $s_V(\lambda)$ have the same order
and they have the  same indicator function
for $0\leq \theta\leq \pi$.  We know that $s_V$ has order at most
$d$ by  \cite[Theorem 7]{zwpf}.  Moreover, for any $M$ chosen
large enough that $s_V$ has no zeros or poles bigger than $M$ on the real
line, for $r>M$ we have 
$$\int_0^r\frac{f'(t)}{f(t)}dt
= \int_M^r \frac{s_V'(t)}{s_V(t)}dt+ O(1). $$  Using 
(\ref{eq:simplebd}) and Lemma \ref{l:derivest}, we see that
$$\int_M^r \frac{s_V'(t)}{s_V(t)}dt= O(r^{d-1})\; \text{as $r\rightarrow 
\infty$}$$
yielding
 \begin{equation}\label{eq:argf}
\int_0^r\frac{f'(t)}{f(t)}dt=O(r^{d-1})\; \text{as $r\rightarrow 
\infty$}.
\end{equation}
  A similar argument gives the same bound for $r\rightarrow -\infty$. 
It remains to check the hypotheses on the indicator function; this is done 
in the next paragraph.

From \cite[Theorem 5]{stefanov}, recalled here
in Theorem \ref{t:stefanov}, for $0\leq \theta\leq \pi$ and large $r$,
$$r^{-d} \ln |f(r e^{i\theta})| \leq a^d h_d(\theta) + o(1)$$
where we have some uniformity in $\theta$-- see Theorem \ref{t:stefanov}.
Thus, using the equation (\ref{eq:frid})
 and (\ref{eq:argf})
$$\lim\sup_{r\rightarrow \infty}r^{-d} N_V(r)= \lim \sup_{r\rightarrow \infty}
r^{-d}\frac{1}{2\pi}\int_0^\pi \ln|f(re^{i\theta})|d\theta\leq 
\frac{a^d}{2\pi}\int_0^\pi h_d(\theta)d\theta.$$
But since  $V\in \mfm_a$,
$$\lim _{r\rightarrow \infty}r^{-d}N_V(r)= \frac{c_d a^d}{d}
=\frac{a^d}{2\pi }\int_0^{\pi}h_d(\theta)
d\theta$$
and
 we see that we must have
$$\lim\sup _{r \rightarrow \infty}r^{-d} \ln |f(r e^{i\theta})| 
= a^d h_d(\theta) \; \text{ for almost every}\;  \theta \in (0,\pi).$$
The left hand side of the above equation is the value of the indicator function
of $f$ at $\theta$.  But the indicator function
 of $f$ is continuous on $(0,\pi)$
\cite[Section 16, point a, page 54]{levin}, and so is $h_d(\theta)$. 
 Thus we must
have 
$$\lim\sup _{r \rightarrow \infty}r^{-d} \ln |f(r e^{i\theta})| 
= a^d h_d(\theta) \; \text{ for}\;  \theta \in (0,\pi).$$

Applying Proposition \ref{p:showcrg} 
to $f(\lambda)$, we see
that $f(\lambda)$ is a function of completely regular growth in 
the upper half plane.  Since $h_d(\theta)$ is a $C^1$ function of $\theta$
for $\theta\in (0,\pi)$, we get the proposition from
Theorem \ref{thm:crg}.

\section{Proof of Theorem \ref{thm:second}}

This section proves Theorem \ref{thm:second}.  We begin by outlining the 
strategy of the proof.

For $0<\varphi<\theta<2 \pi$, 
recall the notation $n_V(r,\varphi,\theta)$ for the 
number of poles of $R_V(\lambda)$ in the sector 
$\{z:\; |z|\leq r ,\; \varphi \leq \arg z \leq \theta \}.$
A representative claim of the theorem is that with $V(z)$, $\Omega$ 
as in the statement of the theorem, $0<\theta<\pi$, 
\begin{multline}\label{eq:onepart}
\int_{\Omega}\psi(z) n_{V(z)}(r,\pi,\theta+\pi)d\mathcal{L}(z) \\
= 
\frac{1}{2\pi d}\left[h_d'(\theta)+ d^2\int_{0}^{\theta} h_d(s) ds\right]a^dr^d 
\int_{\Omega}\psi(z)d\mathcal{L}(z)
+o(r^d)
\end{multline}
as $r\rightarrow \infty$
for any $\psi\in C_c(\Omega)$.  
We prove this via the intermediate step of showing that (\ref{eq:onepart})
holds
for $\psi$ which is the characteristic function of any suitable ball 
in $\Omega$, Proposition \ref{p:lessavg}.  To get (\ref{eq:onepart}) 
for $\psi\in C_c(\Omega)$, we cover the support of $\psi$ with the
union of a finite number of small disjoint balls and a set of small
volume.  On each small ball we can approximate $\psi$ by its value at
the center of the ball and apply Proposition \ref{p:lessavg}.  
This and the necessary
estimates are done in the proof of the theorem which ends this section.

The proof of Proposition \ref{p:lessavg} is done in a number of steps.  
We set
$$N_V(r,\varphi,\theta) = 
\int_0^r\frac{1}{t}(n_V(t,\varphi,\theta)-n_V(0,\varphi,\theta)) dt.$$
Lemma \ref{l:intangles} gives
$
\int_0^{\theta} N_V(r,\pi,\theta'+\pi)d\theta'$
 as a sum of two integrals involving $\ln |s_V|$ and an error of order
$r^{d-1}$.  This follows from an application of one-dimensional complex
analysis,  Lemma \ref{l:derivest} and (\ref{eq:simplebd}).  
Next we consider the function
$$\Psi(z,r,\rho)\defeq \frac{1}{\vol(B(z,\rho))}\int_{z'\in B(z,\rho)} \int_0^\theta
N_{V(z')}(r,\pi,\theta'+\pi)d\theta' d\mathcal{L}(z') .$$
Notice that here we are averaging over balls of varying center $z$.
Fix $\rho$ small, and consider this as a function of $z$ and $r$.  
Lemma \ref{l:intangles} is used to show that
$\Psi$ is the sum of a function $\Psi_1$ which 
is plurisubharmonic in $z$ and a function which is $O(r^{d-1})$.
The proof of Proposition \ref{p:avg1} uses a combination of 
properties of plurisubharmonic functions and the fact
that  $r^{-d} N_{V(z')}(r,\pi,\theta'+\pi)$ is not negative and can be (locally)
 uniformly
bounded above for large $r$ to prove 
an ``averaged'' in $\theta$ and $r$ version of (\ref{eq:onepart})
for $\psi$ the characteristic function of a ball in $\Omega$ satisfying
some conditions.  Propositions \ref{p:avg2} and then \ref{p:lessavg}
eliminate the need to average in $\theta$ and $r$,
using Lemma \ref{l:noavg}.  

The proofs of the other claims of Theorem \ref{thm:second} are quite
similar; the proofs of Proposition \ref{p:totalavg} and the final proof 
of the theorem 
indicate the differences.

Now we turn to proving the theorem.
We shall need an identity related to both (\ref{eq:frid}) and 
to what Levin calls a generalized formula of Jensen 
\cite[Chapter 3, section 2]{levin}.  We define, following \cite{levin}, for a 
 function $f$ meromorphic
in a neighborhood of $\arg z= \theta$ and with $|f(0)|=1$,
\begin{equation}
J^r_f(\theta)\defeq \int_0^r\frac{\ln |f(te^{i\theta})|}{t}dt.
\end{equation}
This integral is well-defined even if $f$ has a zero or pole with
argument $\theta$.

\begin{lemma}\label{l:newgfj}
Let $f$ be holomorphic in $\varphi \leq \arg z \leq \theta$, $|f(0)|=1$, $f$
have no zeros with argument $\varphi$ or $\theta$ and with norm less
than $r$, 
and let $m(r,\varphi,\theta)$ be the number of zeros of 
$f$ in the sector $\varphi < \arg z < \theta$, $|z|\leq r$.  
 Then
\begin{multline}\label{eq:zeroavg}
\int_0^r \frac{m(t,\varphi,\theta)}{t}dt \\= 
\frac{1}{2\pi}\int _0^r \frac{d}{d\theta} J_f^t(\theta)\frac{dt}{t}
+\frac{1}{2\pi} \int _0^r \frac{1}{t}
\int_0^t \frac{d}{ds} \arg f(s e^{i\varphi})ds\; dt
+ \frac{1}{2\pi}
\int_{\varphi}^\theta \ln|f(re^{i\omega})| d\omega.
\end{multline}
\end{lemma}
\begin{proof}
Using the argument principle and 
the Cauchy-Riemann equations just as in \cite[Chapter 3, section 2]{levin}
we see
that
\begin{align*}
2\pi m(r',\varphi, \theta)&=\int_0 ^{r'} \frac{\partial}{\partial t} \arg
f(te^{i\varphi})dt
+ \int _0^{r'} \frac{1}{t} \frac{\partial}{\partial \theta} \ln|f(te^{i\theta})|dt
+ r' \int _\varphi ^\theta \frac{\partial}{\partial r'} \ln |f(r'e^{i \omega})|
d\omega
\end{align*}
when there are no zeros on the boundary of the sector.
As in \cite{levin}, by dividing by $2\pi r'$ and integrating from $0$ to $r$
in $r'$ we obtain the lemma.  
\end{proof}



We note that $|s_V(0)|=1$, since $s_V(\lambda)s_V(-\lambda)=1$.
\begin{lemma}\label{l:intangles}
Suppose $V\in L^{\infty}_{\comp}(\Real^d) $.  Then for $0<\theta<\pi$
$$\int_0^\theta N_V(r,\pi,\theta'+\pi) d \theta'
= \frac{1}{2\pi}\int_0^r J_{s_V}^t(\theta)\frac{dt}{t}
 +\frac{1}{2\pi} \int_0^\theta \int_0 ^{\theta'} \ln |s_V(r e^{i\omega})|d\omega d\theta'
+ O(r^{d-1})$$
as $r\rightarrow \infty.$
The error can be bounded by $c \langle r^{d-1}\rangle $ where the constant
depends only on $\|V\|_\infty$, the support of $V$, and $d$.
\end{lemma}
\begin{proof}
Recall that with at most a finite number of exceptions $\lambda_0$ is a pole of 
$R_V(\lambda)$ if and only if $-\lambda_0$ is a zero of $s_V(\lambda)$, and
the multiplicities coincide.  As in the proof of Proposition \ref{p:angles},
if $s_V(\lambda)$ has poles $\lambda_1,...,\lambda_m$ in the closed
upper half plane we introduce the function 
$$f(\lambda)=\frac{(\lambda-\lambda_1)\cdot \cdot \cdot (\lambda-\lambda_m)}
{(\lambda+\lambda_1)\cdot \cdot \cdot (\lambda+\lambda_m)}s_V(\lambda)
$$
which is holomorphic in the closed upper half plane. 
The poles of $s_V$ in the closed upper half plane correspond to eigenvalues
and the number of such poles can be bounded by a constant depending on $d$,
$\|V\|_\infty$,  and the support of $V$.
 Note that
$f$ has no 
zeros on the real line.  
Moreover, $\ln|f(re^{i\theta})|= \ln|s_V(re^{i\theta})|+O(1)$ for $r
\rightarrow \infty$, $0\leq \theta\leq \pi$ and that $s_V$ and $f$ 
have all but finitely many of the same zeros.

Choose $0<M<\infty$ so that $s_V(\lambda)$ has no zeros in the upper
half plane with norm less than or equal to $M$.  This constant $M$
can be chosen to depend only on $\|V\|_{\infty}$, the support of $V$,
and $d$.  Now by using the 
relationship between the poles of $R_V(\lambda)$ and
the zeros of $s_V=\det S_V$, the relationships between $f$ and 
$s_V$ just mentioned, and applying Lemma \ref{l:newgfj} to $f$ we
see that for $r>M$, $0<\theta'<\pi,$
\begin{multline}\label{eq:int}
N_V(t,\pi,\theta'+\pi)
= \frac{1}{2\pi} \int_M^r\frac{\partial}{\partial \theta'}
J^t_{s_V}(\theta') \frac{dt}{t} 
+\frac{1}{2\pi}\int_M^r\frac{1}{t}\int_M^t
\frac{d}{dt'} \arg s_V(t')dt' dt \\
+ \frac{1}{2\pi}\int_0^{\theta'}\ln|s_V(r e^{i\omega})|d\omega +O((\ln r)^2)
\end{multline}
if $f$ has no zeros with argument $\theta'$ and norm not exceeding $r$.
Here we are using that $\int_0^M\frac{\partial}{\partial \theta'}
J^t_{f}(\theta') \frac{dt}{t}=O(1)$
and 
\begin{multline*}
\int_0^r\frac{1}{t}\int_0^t
\frac{d}{dt'} \arg f(t')dt' dt \\= \int_M^r \frac{1}{t}\int_M^t
\frac{d}{dt'} \arg f(t')dt' dt + \int_M^r \frac{1}{t}\int_0^M
\frac{d}{dt'} \arg f(t')dt' dt
+ \int_0^M \frac{1}{t}\int_0^t
\frac{d}{dt'} \arg f(t')dt' dt.
\end{multline*}
But $\int_M^r \frac{1}{t}\int_0^M
\frac{d}{dt'} \arg f(t')dt' dt= O(\ln r)$ and
 $\int_0^M \frac{1}{t}\int_0^t
\frac{d}{dt'} \arg f(t')dt' dt= O(1)$.   Additionally, for $t\rightarrow
\infty$,
$\frac{d}{dt} \arg f(t)=\frac{d}{dt}\arg s_V(t)+O(1/t)$.  These remainders
can be bounded using constants depending only on $\|V\|_{\infty}$, 
$\supp V$, and $d$.

Notice that for fixed value of $r>M$ there are only finitely many
values of $\theta'$ with $s_V$ having a zero with argument $\theta'$ and
norm at most $r$.  
We integrate (\ref{eq:int}) in $\theta'$ from $0$ to $\theta$
and, as in the proof of Jensen's equality, use the fact 
that both sides of the equation
below  are continuous functions of
$\theta$, to get
\begin{multline*}
\int_0^\theta N_V(r,\pi,\theta'+\pi) d \theta'
= \frac{1}{2\pi}\int_M^r J_{s_V}^t(\theta)\frac{dt}{t} 
- \frac{1}{2\pi} \int_M^r J^t_{s_V}(0)\frac{dt}{t}  \\
+ \frac{\theta}{2\pi}\int_M^r\frac{1}{t} \int_M^t
\frac{d}{dt'} \arg s_V(t')dt' dt
+\frac{1}{2\pi} \int_0^\theta \int_0 ^{\theta'} 
\ln |s_V(r e^{i\omega})|d\omega d\theta' +O((\ln r)^2).
\end{multline*}

The bounds of Lemma \ref{l:derivest} and (\ref{eq:simplebd}) mean 
that as $r\rightarrow \infty$
$$\frac{1}{2\pi} \int_M^r J^t_{s_V}(0)\frac{dt}{t}= O(r^{d-1})$$
and $$\frac{\theta}{2\pi}\int_M^r\frac{1}{t} \int_M^t
\frac{d}{dt'} \arg s_V(t')dt' dt =O(r^{d-1})$$
where the bounds can be made uniform in $V$ with support contained
in a fixed compact set and $\|V\|_{\infty}$ bounded.
Moreover, we note that
$\int_0^M J^t_{S_V}(\theta)\frac{dt}{t}=O(1).$ 
\end{proof}

We shall need some notation for the results which follow.   
Let  $\Omega\subset \Complex^{d'}$ be an open set containing a point $z_0$. 
For
$\rho>0$ small enough that $B(z_0,\rho)\subset \Omega$ we define
$\Omega_\rho$ to be the connected component of 
$\{z\in \Omega: \dist(z,\Omega^c)> \rho\}$ which contains $z_0$.

\begin{prop}\label{p:avg1}
Let $V$, $z_0$, $\Omega$ satisfy the assumptions of Theorem \ref{thm:first},
let $\rho>0$ be small enough that $B(z_0,2\rho)\subset \Omega$, and let
$\Omega_\rho$ be as defined above.
Then for $z\in \Omega_{2\rho}$, $0<\theta<\pi$, 
\begin{multline*}
\Psi(z,r,\rho)\defeq \frac{1}{\vol(B(z,\rho))}\int_{z'\in B(z,\rho)} \int_0^\theta
N_{V(z')}(r,\pi,\theta'+\pi)d\theta' d\mathcal{L}(z') \\
= \frac{1}{2\pi}a^d r^d\left( \frac{1}{d^2}h_d(\theta)+\int_0^\theta 
\int_0^{\theta'} h_d(\omega)d\omega d\theta'\right) +o(r^d)
\end{multline*}
as $r\rightarrow \infty$.
\end{prop}
\begin{proof}
First note that since $0\leq d N_{V(z)}(z,\pi,\theta+\pi)\leq c_dr^da^d+o(r^d)$,
and the bound is uniform on compact sets of $z$, we get that holding $\rho$
fixed, 
$r^{-d}\Psi(\bullet,r,\rho)$ is a family uniformly continuous in $z$ for
$z$ in compact sets of $\overline{\Omega}_{2\rho}$. 

We shall use Lemma \ref{l:intangles}.  Note that by Stefanov's results
recalled in Theorem \ref{t:stefanov}, for large $r$
$$\frac{1}{2\pi}\int_0^r J^t_{s_{V(z)}}(\theta)\frac{dt}{t}
\leq \frac{1}{2\pi} \frac{1}{d^2}h_d(\theta) a^d r^d+o(r^d)$$
where the term $o(r^d)$ can be bounded uniformly in $z$ in 
compact sets of $\overline{\Omega}_{\rho}$.  Recall that
this is a statement about large $r$ behavior, and holds even if
$s_V(z)$ has poles in the upper half plane, since it has at most finitely 
many.  By the same 
argument, for large $r$
$$\int_0^\theta\int_0^{\theta'} \ln |s_{V(z)}(re^{i\omega})|d\omega d\theta'
\leq \int _0 ^\theta\int_0^{\theta'} h_d(\omega)d\omega d\theta' a^d r^d
+o(r^d).$$

Using Lemma \ref{l:intangles}, we find that
\begin{multline*}
\Psi(z,r,\rho)= \frac{1}{2\pi \Vol(B(z,\rho))}\int_{z'\in B(z,\rho)}
\int_0^r J_{s_V(z')}^t(\theta)\frac{dt}{t} d\mathcal{L}(z') \\ 
 +\frac{1}{2\pi\Vol(B(z,\rho))} \int_{z'\in B(z,\rho)}\int_0^\theta \int_0 ^{\theta'} \ln |s_{V(z')}(r e^{i\omega})|d\omega d\theta'd\mathcal{L}(z')
+ O(r^{d-1}).
\end{multline*}  
Let $M= 2 \alpha_d \max_{z\in \overline{\Omega_\rho}}\| V(z) \|_{\infty}$
and set, for $r>M$,
 \begin{multline*}
\Psi_1(z,r,\rho)=  \frac{1}{2\pi \Vol(B(z,\rho))}\int_{z'\in B(z,\rho)}
\int_{M}^r J_{s_{V(z')}}^t(\theta)\frac{dt}{t} d\mathcal{L}(z') \\
 +\frac{1}{2\pi \Vol(B(z,\rho))} \int_{z'\in B(z,\rho)}\int_0 ^\theta 
\int_0 ^{\theta'} \ln |s_{V(z')}(r e^{i\omega})|d\omega d\theta'd\mathcal{L}(z')
\end{multline*}
and note that 
$$\Psi(z,r,\rho)=\Psi_1(z,r,\rho)+O(r^{d-1}).$$
By the bounds above, 
\begin{equation}\label{eq:oneupperbd}
 \Psi_1(z,r,\rho)\leq \frac{1}{2\pi}\left( \frac{1}{d^2}h_d(\theta)
+ \int _0^\theta\int_0^{\theta'} h_d(\omega)d\omega d\theta'\right) a^d r^d
+o(r^d).
\end{equation}

Using \cite[Proposition I.14]{l-g} and
the fact that $\ln|s_{V(z)}(\lambda)|$ is a plurisubharmonic function
of $z\in \Omega$ when $|\lambda|>2 \alpha_d \|V(z)\|_{\infty}$ and $\lambda $ 
lies in the upper half plane, we see that 
$\Psi_1(z,r,\rho)$ is a plurisubharmonic function of $z\in \Omega_{2\rho}$.
Since by Proposition \ref{p:showcrg} $s_{V(z_0)}(\lambda)$ is of completely
regular growth in 
$0<\arg \lambda <\pi$, using Lemma \ref{l:intangles}
 and \cite[Chapter III, Sec. 2,
Lemma 2]{levin},
$$\lim_{r\rightarrow \infty} r^{-d} 
\int_0^{\theta}N_{V(z_0)}(r,\pi,\theta'+\pi) d\theta'
=\frac{1}{2\pi} \left(\frac{1}{d^2} h_d(\theta)+\int_0^\theta\int_0^{\theta'}
h_d(\omega)d\omega d\theta'\right) a^d.$$
By the most basic property of plurisubharmonic functions,
$$\Psi_1(z_0,r,\rho)\geq \frac{1}{2\pi}
\int_{M}^r J^t_{s_V(z_0)}(\theta)\frac{dt}{t} + \frac{1}{2\pi}\int_0^\theta
\int_0^{\theta'}\ln |s_{V(z_0)}(re^{i\omega})|d\omega d\theta'.$$
But the right hand side of this equation is 
$\int_0^\theta N_{V(z_0)}(r,\pi,\theta'+\pi)d\theta' +O(r^{d-1})$,
so we see that 
$$\lim \inf _{r\rightarrow \infty} 
r^{-d} \Psi_1(z_0,r,\rho)\geq  \frac{1}{2\pi} \left(\frac{1}{d^2} h_d(\theta)+\int_0^\theta \int_0^{\theta'}
h_d(\omega)d\omega d\theta'\right) a^d.$$
Combining this with (\ref{eq:oneupperbd}), we find
\begin{equation}
\lim _{r\rightarrow \infty} 
r^{-d} \Psi_1(z_0,r,\rho)
= \frac{1}{2\pi} \left(\frac{1}{d^2} h_d(\theta)+\int_0^\theta\int_0^{\theta'}
h_d(\omega)d\omega d\theta'\right) a^d.
\end{equation}
Using this and the upper bound (\ref{eq:oneupperbd}) on $\Psi_1$, since
$\Psi_1$ is plurisubharmonic in $z$ it follows from \cite[Proposition 1.39]{l-g}
(recalled here in Proposition \ref{p:l-g})
 that for any sequence
$\{r_j\}$, $r_j\rightarrow \infty$ there is a pluripolar set 
$E\subset \Omega_\rho$ (which may depend on the sequence) so that 
$$\lim \sup _{j\rightarrow \infty}r_j^{-d} \Psi_1(z,r_j,\rho)
= \frac{1}{2\pi} \left(\frac{1}{d^2} h_d(\theta)+\int_0^\theta\int_0^{\theta'}
h_d(\omega)d\omega d\theta'\right) a^d$$
for all $z\in \Omega_\rho \setminus E$.  Since $\lim_{r\rightarrow \infty}
r^{-d}(\Psi_1(z,r,\rho)-\Psi(z,r,\rho))=0$, the same conclusion holds
for $\Psi$ in place of $\Psi_1$.

Suppose there is some $z_1\in \Omega_\rho$ and some sequence $r_j\rightarrow
\infty$ so that 
$$\lim_{j\rightarrow \infty} r_j^{-d}\Psi(z_1,r_j,\rho) <\frac{1}{2\pi} \left(\frac{1}{d^2} h_d(\theta)+\int_0^\theta \int_0^{\theta'}
h_d(\omega)d\omega d\theta'\right) a^d.$$
Then, using the uniform continuity of $r^{-d}\Psi(z,r,\rho)$ in $z$, we find
there must be an $\epsilon>0$ so that 
$$ \lim \sup _{j\rightarrow \infty} r_j^{-d}\Psi(z,r_j,\rho) <\frac{1}{2\pi} \left(\frac{1}{d^2} h_d(\theta)+\int_0^\theta \int _0^{\theta'}
h_d(\omega)d\omega d\theta'\right) a^d
$$
for all $z\in B(z_1,\epsilon)$.  But since $B(z_1,\epsilon)$ is not contained
in a pluripolar set, we have a contradiction.  Thus
$$
\lim_{r\rightarrow \infty} r^{-d}\Psi(z,r,\rho) =
\frac{1}{2\pi} \left(\frac{1}{d^2} h_d(\theta)+\int_0^\theta\int_0^{\theta'}
h_d(\omega)d\omega d\theta'\right) a^d$$
for all $z\in \Omega_\rho$.
\end{proof}

The following lemma will be used to remove the need to 
average in $\theta$ as in Proposition \ref{p:avg1}.
\begin{lemma}\label{l:noavg}
Let $M(r,\theta)$ be a function so that for any fixed
positive $r_0>C_0$, $M(r_0,\theta)$ is a
non-decreasing function of $\theta$, and suppose 
$$\lim_{r\rightarrow \infty }r^{-d} \int_0^{\theta}M(r,\theta')d\theta'
=\alpha(\theta)$$
for $\theta_1<\theta<\theta_2$.  Then if $\alpha$ is differentiable
at $\theta$, then
$$\lim_{r\rightarrow \infty }r^{-d} M(r,\theta)=\alpha'(\theta).$$
\end{lemma}
\begin{proof}
Let $\epsilon >0$.  Then, since $M(r,\theta)$ is non-decreasing in 
$\theta$, 
$$\int_0^{\theta+\epsilon}M(r,\theta')d\theta' -
\int_0^{\theta}M(r,\theta')d\theta' \geq \epsilon M(r,\theta)$$
which, under rearrangement, yields
$$r^{-d} M(r,\theta)\leq r^{-d}
\frac{ \int_0^{\theta+\epsilon} M(r,\theta')d\theta' - 
\int_0^{\theta} M(r,\theta')d\theta'}{\epsilon}.$$
Thus
$$\lim \sup _{r\rightarrow \infty}r^{-d} M(r,\theta)\leq 
\frac{\alpha(\theta+\epsilon)-\alpha(\theta)}{\epsilon}.$$
Likewise, we find
$$\lim \inf _{r\rightarrow \infty}r^{-d} M(r,\theta)\geq 
\frac{\alpha(\theta)-\alpha(\theta-\epsilon)}{\epsilon}.$$
Since both these equalities must hold for all $\epsilon>0$, the
lemma follows from the assumption 
that $\alpha$ is differentiable at $\theta$.
\end{proof}

The following proposition follows from Proposition \ref{p:avg1},
but is stronger as it does not require averaging in the $\theta'$ 
variables.
\begin{prop} \label{p:avg2}
Let $V$, $z_0$, $\Omega$ satisfy the assumptions of Theorem \ref{thm:first}, 
and $\rho>0$, $\Omega_\rho$ be as in Proposition \ref{p:avg1}.
Then for $z\in \Omega_{2 \rho}$, $0<\theta<\pi$, as $r\rightarrow \infty$
\begin{multline*}
\frac{1}{\Vol(B(z,\rho))}\int_{z'\in B(z,\rho)} 
N_{V(z')}(r,\pi,\theta+\pi)d\mathcal{L}(z') \\ 
= \frac{1}{2\pi}a^d r^d\left( \frac{1}{d^2}h_d'(\theta)+
\int_0^{\theta} h_d(\omega)d\omega \right) +o(r^d).
\end{multline*}
\end{prop}
\begin{proof}
This follows from applying Lemmas \ref{l:noavg}
and \ref{l:hdbasics} to the results of Proposition 
\ref{p:avg1}.
\end{proof}

Proposition \ref{p:avg2} does not give results for the counting 
function for all the resonances (note that we cannot have $\theta=\pi$).
The following fills this gap.
\begin{prop} \label{p:totalavg}
Let $V$, $z_0$, $\Omega$ satisfy the assumptions of Theorem \ref{thm:first},
 and 
$\rho>0$,  $\Omega_\rho$ as in Proposition \ref{p:avg1}.
Then for $z\in \Omega_{2\rho}$, as $r\rightarrow \infty$
$$
\frac{1}{\Vol(B(z,\rho))}\int_{z'\in B(z,\rho)} 
N_{V(z')}(r)d\mathcal{L}(z')
= \frac{1}{2\pi}a^d r^d
\int_0^{\theta} h_d(\omega)d\omega  +o(r^d).$$
\end{prop}
\begin{proof}
The proof of this is very similar to that of Proposition \ref{p:avg1}.
In fact, the main difference is the use of (\ref{eq:frid}), which together
with Lemma \ref{l:derivest} and (\ref{eq:simplebd})  gives us by handling
possible poles in the upper half plane using 
 a method similar to the proof of Lemma \ref{l:intangles},
$$\frac{1}{\Vol(B(z,\rho))}\int_{z'\in B(z,\rho)} 
N_{V(z')}(r)d\mathcal{L}(z')= \Psi_1(z,r,\rho)+O(r^{d-1})$$
where
$$\Psi_1(z,r,\rho)= \frac{1}{\Vol(B(z,\rho))}\frac{1}{2\pi} 
\int_{z'\in B(z,\rho)}\int_0^{\pi} \ln |s_{V(z')}(r e^{i\theta})|d\theta d\mathcal{L}(z').
$$
Using that $\Psi_1$ is plurisubharmonic in $z$, the proof now follows
just as in Proposition \ref{p:avg1}.
\end{proof}

The following proposition is much like
 Propositions \ref{p:avg2} and \ref{p:totalavg}, but eliminates the 
 average in the $r$ variable.
\begin{prop}\label{p:lessavg}
Let $V,\; \Omega,\; z_0$ satisfy the conditions of 
Theorem \ref{thm:first}, 
and let $\rho$, $\Omega_\rho$ be as in Proposition \ref{p:avg1}.
Then for 
 $0<\theta<\pi$, $z\in \Omega_{2\rho}$,
$$\frac{1}{\Vol(B(z,\rho))}\int_{z'\in B(z,\rho)}n_{V(z')}(r,\pi,\theta+\pi)
d\mathcal{L}(z')= \frac{a^dr^d}{2\pi} \left(\frac{1}{d}h_d'(\theta)
+ d \int_0^\theta h_d(\theta)d\theta\right) +o(r^d)
$$
and 
$$\frac{1}{\Vol(B(z,\rho))}\int_{z'\in B(z,\rho)}n_{V(z')}(r) d\mathcal{L}(z')
=\frac{d}{2\pi} a^d r^d \int_0^\pi h_d(\theta)d\theta + o(r^d)$$
as $r\rightarrow \infty$.
\end{prop}
\begin{proof}
This proof follows from Propositions \ref{p:avg2} and \ref{p:totalavg},
using, in addition, a result like that of \cite[Lemma 1]{stefanov}
or Proposition \ref{l:noavg}.
\end{proof}

\noindent {\em Proof of Theorem \ref{thm:second}}.
Let $M=\max(1+|\psi(z)|)$, and for $\rho>0$ small enough
that $B(z_0,\rho)\subset \Omega$, set 
$\Omega_\rho$ to be the connected component of $\{z\in \Omega:
\dist(z,\Omega^c)>\rho\}$ which contains $z_0$. 
Given $\epsilon>0$, choose $\rho>0$ such that 
$B(z_0 , 2\rho)\subset \Omega$ and so that 
\begin{equation}
\vol(\supp \psi \cap (\Omega \setminus \Omega_{2\rho}))<\frac{\epsilon}
{10 M e(c_da^d+1)}.
\end{equation}
Since $\psi$ is continuous
with compact support, we can find a $\delta_1>0$, $\delta_1<\rho$
so that if $|z-z'|<\delta_1$, then 
$|\psi(z)-\psi(z')|<\frac{\epsilon}{10e(1+\vol \supp\psi)  (a^dc_d+1) }$.
We may find a finite number $J$ of disjoint balls $B(z_j,\epsilon_j)$ 
so that $\epsilon_j<\delta_1$, $z_j\subset \Omega_{2\rho}$, and  
$$\vol\left(\supp \psi \setminus (\cup_1^J B(z_j,\epsilon_j)\right)+ \vol
\left(\cup_1^J B(z_j,\epsilon_j) \setminus \supp \psi\right)
<\frac{\epsilon}{4 M e(a^d c_d+1)}.$$

Let $\pi\leq \varphi'\leq \theta'\leq 2\pi$.
Now 
\begin{align*}
& \int \psi(z) n_{V(z)}(r,\varphi',\theta') d\mathcal{L}(z) \\
& = \sum_{j=1}^J \int_{B(z_j,\epsilon_j)}
 \psi(z) n_{V(z)}(r,\varphi',\theta') d\mathcal{L}(z)
+ \int_{\supp \psi \setminus(\cup B(z_j,\epsilon_j))}
 \psi(z) n_{V(z)}(r,\varphi',\theta') d\mathcal{L}(z).
\end{align*}
We will use that the bound (\ref{eq:stupperbd}) implies that
$n_V(z)\leq e c_d a^d r^d +o(r^d).$
By our choice of $B(z_j,\epsilon_j)$, 
$$\left|\int_{\supp \psi \setminus(\cup B(z_j,\epsilon_j))}
 \psi(z) n_{V(z)}(r,\varphi',\theta') d\mathcal{L}(z)\right| 
\leq \frac{\epsilon}{4}(r^d +o(r^d)).$$
By our choice of $\delta_1$ and
the assumption that  $\epsilon_j<\delta_1$, we have
\begin{multline*}
\left| \sum_{j=1}^J \int_{B(z_j,\epsilon_j)}
 \psi(z) n_{V(z)}(r,\varphi',\theta') d\mathcal{L}(z)
- \sum_{j=1}^J \int_{B(z_j,\epsilon_j)}
 \psi(z_j) n_{V(z)}(r,\varphi',\theta') d\mathcal{L}(z) \right| \\
\leq \frac{\epsilon}{5}(r^d+o(r^d)).
\end{multline*}
By Proposition \ref{p:lessavg}, if $0<\theta<\pi$,
\begin{multline*}
\sum_{j=1}^J \int_{B(z_j,\epsilon_j)}
 \psi(z_j) n_{V(z)}(r,\pi,\pi+\theta) d\mathcal{L}(z)
\\= \left( \sum_{j=1}^J 
 \psi(z_j)\vol(B(z_j,\epsilon_j))\right)
\frac{1}{2\pi}a^d r^d\left(\frac{1}{d} h_d'(\theta)+ 
d \int_0^{\theta} h_d(\omega)d\omega \right) + o(r^d),
\end{multline*}
and 
$$ \sum_{j=1}^J \int_{B(z_j,\epsilon_j)}
 \psi(z_j) n_{V(z)}(r)d\mathcal{L}(z)
= \left( \sum_{j=1}^J 
 \psi(z_j)\vol(B(z_j,\epsilon_j))\right)
\frac{d}{2\pi}a^d r^d \int_0^\pi h_d(\omega)d\omega + o(r^d).
$$
Again using our choice of $\delta_1$, $z_j$, and $\epsilon_j$, we have
$$ \left| \sum_{j=1}^J 
 \psi(z_j)\vol(B(z_j,\epsilon_j)) -
\int \psi(z)d\mathcal{L}(z)\right|<\frac{2\epsilon}{5(c_da^d+1)}.$$
Thus we have shown that given $\epsilon>0$, 
if $0<\theta<\pi$,
\begin{multline}
\left|\int \psi(z) n_{V(z)}(r,\pi,\theta+\pi) d\mathcal{L}(z)
-
\frac{a^d r^d}{2\pi}\int\psi(z)d\mathcal{L}(z)\left(\frac{1}{d} h_d'(\theta)+ 
d \int_0^{\theta} h_d(\omega)d\omega \right)\right| \\
\leq  \epsilon r^d +o(r^d)
\end{multline}
and 
\begin{equation}
\left|\int \psi(z) n_{V(z)}(r ) d\mathcal{L}(z)
- c_d a^d r^d\int \psi(z)d\mathcal{L}(z)
 \right|
\leq \epsilon r^d +o(r^d).
\end{equation}
Thus we have 
proved the first and third statements of the theorem.  The second statement
of the theorem follows from the other two.

\section{Proof of Theorem \ref{thm:first}}
\label{s:proofthmfirst}

This proof uses some ideas similar to those used in the proofs of 
Propositions \ref{p:avg1} and \ref{p:totalavg}.  In fact, because the
proofs are so similar we shall only give an outline.

Note that by (\ref{eq:frid}), (\ref{eq:simplebd}),
and 
and    Lemma \ref{l:derivest},
using an argument similar to the proofs of Lemma \ref{l:intangles}
and Proposition \ref{p:avg1}, 
$$N_{V(z)}(r)=\Psi(z,r)+o(r^{d-1})$$
where
$$\Psi(z,r)=\frac{1}{2\pi} \int_0^\pi \ln |s_{V(z)}(re^{i\theta})|d\theta$$
is, for fixed (large) $r$ a plurisubharmonic function of $z\in \tilde{\Omega}
\Subset\Omega$.  
Since 
$$\lim \sup_{r\rightarrow \infty} r^{-d} \Psi(z,r)
\leq  \frac{a^d}{2\pi}\int_0^\pi h_d(\theta)d\theta$$ and this 
maximum is achieved at $z=z_0\in \Omega$, we get the first part of the Theorem
by applying \cite[Proposition 1.39]{l-g}, recalled in Proposition 
\ref{p:l-g}.

To obtain the second part, note that as in the proof of Proposition
\ref{p:avg1}, for $0<\theta<\pi$,
$$\int_0^\theta N_{V(z)}(r,\pi,\theta'+\pi)d\theta'
= \Psi_2(z,r,\theta)+o(r^d)$$
where 
$$\Psi_2(z,r,\theta)=\frac{1}{2\pi}
\int_M^r J_{s_V(z)}^t(\theta)\frac{dt}{t}  \\
 +\frac{1}{2\pi} \int_0 ^\theta 
\int_0 ^{\theta'} \ln |s_{V(z)}(r e^{i\omega})|d\omega 
d\theta'.$$
Since this is a plurisubharmonic function of $z\in \tilde{\Omega}$,
$\tilde{\Omega}\Subset \Omega$,
if $M$ is chosen so that 
$M\geq  2 \alpha_d \max_{z\in \overline{\tilde{\Omega}}}\| V \|_{\infty}$,
a similar argument as in
the proof of Proposition \ref{p:avg1}
 shows that there exists  a pluripolar set
$E_\theta \subset \Omega$ so 
that 
$$2\pi \lim \sup_{r\rightarrow \infty}r^{-d}\Psi_2(z,r,\theta)
=a^d\left( \frac{1}{d^2}h_d(\theta)+\int_0^\theta \int_0^{\theta'}h_d(\omega)
d\omega d\theta'\right)$$
for all $z\in \Omega \setminus E_\theta$.  Again, we use that 
this equality holds when $z=z_0$.
Note that if the second part of the theorem can be proved for 
a small $\theta_0$, it is proved for all $\theta $ with $\theta \geq \theta_0$.
Thus, it is most interesting for small $\theta$.  Choose $\theta>0$ 
sufficiently small that $h_d(\theta)\geq \theta h_d'(0+)/2$, where
we denote $\lim _{\epsilon \downarrow 0 } h_d'(\epsilon)=h_d'(0+).$  Note
that $h_d(\theta)\geq 0$.  
Now, if 
\begin{align*}
\lim \sup_{r\rightarrow \infty} r^{-d} \int_0^\theta
N_V(r,\pi,\pi+\theta')d\theta'&  = \frac{a^d}{2\pi}
\left( \frac{1}{d^2} h_d(\theta)+\int_0^\theta\int_0^{\theta'} h_d(\omega)
d\omega d\theta'
\right)\\
&  \geq \frac{a^d}{4\pi d^2} h_d'(0+)\theta,
\end{align*}
then since $N_V(r,\pi,\pi)=O(1)$, we must have 
$$\lim \sup_{r\rightarrow \infty} r^{-d} 
N_V(r,\pi,\pi+\theta)  \geq \frac{a^d}{4\pi d^2} h_d'(0+).$$

\end{document}